\documentclass[11pt]{amsart}

\usepackage[utf8]{inputenc}
\usepackage[T1]{fontenc}

\usepackage{mathrsfs,amssymb}
\usepackage{graphicx}
\begin{document}

\newtheorem{theorem}{Theorem}[section]
\newtheorem{proposition}[theorem]{Proposition}
\newtheorem{lemma}[theorem]{Lemma}
\newtheorem{corollary}[theorem]{Corollary}
\newtheorem{conjecture}[theorem]{Conjecture}
\newtheorem{question}[theorem]{Question}
\newtheorem{problem}[theorem]{Problem}
\theoremstyle{definition}
\newtheorem{definition}{Definition}

\theoremstyle{remark}
\newtheorem{remark}[theorem]{Remark}

\renewcommand{\labelenumi}{(\roman{enumi})}
\def\theenumi{\roman{enumi}}

\numberwithin{equation}{section}

\renewcommand{\Re}{\operatorname{Re}}
\renewcommand{\Im}{\operatorname{Im}}

\def\scrA{{\mathcal A}}
\def\scrB{{\mathcal B}}
\def\scrD{{\mathcal D}}
\def\scrL{{\mathcal L}}
\def\scrS{{\mathcal S}}

\def \G {{\Gamma}}
\def \g {{\gamma}}
\def \R {{\mathbb R}}
\def \C {{\mathbb C}}
\def \Z {{\mathbb Z}}
\def \Q {{\mathbb Q}}
\def \TT {{\mathbb T}}
\newcommand{\T}{\mathbb T}
\def \GinfmodG {{\Gamma_{\!\!\infty}\!\!\setminus\!\Gamma}}
\def \GmodH {{\Gamma\setminus\H}}
\def \vol {\hbox{vol}}
\def \sl  {\hbox{SL}_2(\mathbb Z)}
\def \slr  {\hbox{SL}_2(\mathbb R)}
\def \psl  {\hbox{PSL}_2(\mathbb R)}

\newcommand{\mattwo}[4]
{\left(\begin{array}{cc}
                        #1  & #2   \\
                        #3 &  #4
                          \end{array}\right) }

\newcommand{\rum}[1] {\textup{L}^2\left( #1\right)}
\newcommand{\norm}[1]{\left\lVert #1 \right\rVert}
\newcommand{\abs}[1]{\left\lvert #1 \right\rvert}
\newcommand{\inprod}[2]{\left \langle #1,#2 \right\rangle}
 
\renewcommand{\^}[1]{\widehat{#1}}

\renewcommand{\i}{{\mathrm{i}}}

\newcommand{\area}{\operatorname{area}}
\newcommand{\ecc}{\operatorname{ecc}}

\newcommand{\Op}{\operatorname{Op}}
\newcommand{\dom}{\operatorname{Dom}}
\newcommand{\Dom}{\operatorname{Dom}}
\newcommand{\tr}{\operatorname{tr}}

\newcommand{\Norm}{\mathcal N}
\newcommand{\simgeq}{\gtrsim}%
\newcommand{\simleq}{\lesssim}
%%these are with amssym, else use {\stackrel{>}{\sim}}

\newcommand{\UN}{U_N}
\newcommand{\OPN}{\operatorname{Op}_N}
\newcommand{\HN}{\mathcal H_N}
\newcommand{\TN}{T_N}  %{\hat T_N}
\newcommand{\PDO}{\Psi\mbox{DO}}

\newcommand{\intinf}{\int_{-\infty}^\infty}
\newcommand{\lcm}{\operatorname{lcm}}

\title{Small gaps in the spectrum of the rectangular billiard}

\author[V. Blomer, J. Bourgain, M. Radziwi{\l\l}  and Z. Rudnick]
{Valentin Blomer, Jean Bourgain, Maksym Radziwi{\l\l}   and Ze\'ev Rudnick}

\address{Mathematisches Institut, Bunsenstr. 3-5, 37073 G\"ottingen, Germany}
\email{vblomer@math.uni-goettingen.de}

\address{Institute for Advanced Study, 1 Einstein Drive, Princeton NJ 08540, USA}
\email{bourgain@ias.edu}

\address{Department of Mathematics, McGill University, 845 Rue Sherbrooke Ouest, Montr\'eal, Qu\'ebec H3A 0G4, Canada}
%\address{Department of Mathematics, Rutgers University, 110 Frelinghuysen Rd., Piscataway, NJ
%08854-8019}
\email{maksym.radziwill@gmail.com}

\address{Raymond and Beverly Sackler School of Mathematical Sciences,
Tel Aviv University, Tel Aviv 69978, Israel}
\email{rudnick@post.tau.ac.il}

\begin{abstract}We study 
the size of the minimal gap  between the first $N$ eigenvalues of the Laplacian on a rectangular billiard having irrational squared aspect ratio $\alpha$, in comparison to the corresponding quantity for a Poissonian sequence. If $\alpha$ is a quadratic irrationality of certain type, such as the square root of a rational number, we show that the  minimal gap is roughly of size  $1/N$, which is essentially consistent with Poisson statistics. We also give  related results for a set of $\alpha$'s of full measure. However, on  a fine scale we show that Poisson statistics is violated for  all $\alpha$. The proofs use a variety of ideas of an  arithmetical nature, involving Diophantine approximation, the theory of continued fractions, and results in analytic number theory. 
\end{abstract}


\thanks{V.B. was supported in part by the Volkswagen Foundation and NSF grant 1128155 while enjoying the hospitality of the Institute for Advanced Study.  J.B. was partially supported by NSF grant  DMS-1301619. M.R. is funded by the National Science and
Engineering Research Council of Canada. 
Z.R. was supported by the Friends of the Institute for Advanced Study, 
and by the European Research Council under the European Union's Seventh
Framework Programme (FP7/2007-2013)/ERC grant agreement
n$^{\text{o}}$ 320755. The United States Government is authorized to reproduce and distribute reprints notwithstanding any copyright notation herein. }

\date{\today}
\maketitle

\section{Introduction} 
%We consider a flat rectangular torus $\T^2_\alpha=\R^2/2\pi (\frac 1{\sqrt{\alpha}} \Z\oplus\Z)$,  (say $\alpha\geq 1$) which has as a fundamental domain a rectangle of width $2\pi/\sqrt{\alpha}$ and height $2\pi$, with aspect ratio $\sqrt{\alpha}$  and area $4\pi^2/\sqrt{\alpha}$.  The spectrum of the flat Laplacian on $\T^2_\alpha$ consists of the numbers $\lambda_{m,n} = \alpha m^2+n^2$, with $m,n\in \Z$. 
%Weyl's law for the spectrum hence says that 
%$$N(X):=\#\{\lambda\leq X\}\sim \frac{\pi}{\sqrt{\alpha}} X $$

The %\marginpar{we may or may not say something about methods in the abstract; it's nice on the one hand, but the abstract is already quite long}   
local statistics of the energy levels of several integrable systems are believed to follow Poisson statistics \cite{BT}. In this note we examine a variant of these statistics,  the size of the minimal gap between levels,  for the energy levels of a particularly simple  system,   
a rectangular billiard. If the rectangle has width $ \pi/\sqrt{\alpha}$ and height $ \pi$, with aspect ratio $\sqrt{\alpha}  $, then  the energy levels, meaning the eigenvalues of the Dirichlet Laplacian, consist  of the numbers $  \alpha m^2+n^2$ with integers $m,n\geq 1$.  

The case of rational  $\alpha$ is special: The eigenvalues lie in a lattice,    in particular the nonzero gaps are bounded away from zero, and there are arbitrarily large multiplicities. We exclude this case from our discussion.  
If $\alpha$ is irrational, we get a simple spectrum $ 0<\lambda_1<\lambda_2<\cdots$,  with growth (Weyl's law)
$$\#\{j:\lambda_j\leq X\}=\#\{(m,n): m,n\geq 1, \;  \alpha m^2+n^2\leq X\}\sim  \frac{\pi}{4\sqrt{\alpha}} X  
$$
as $X \rightarrow \infty$. 
In this setting, the pair correlation function has been shown to be Poissonian \cite{ EMM}  for Diophantine $\alpha$, 
see also \cite{Sarnak} for a related problem. 

We wish to study the size of the minimal gap function of the spectrum, defined as 
$$ \delta_{\min}^{(\alpha)}(N)  = \min(\lambda_{i+1}-\lambda_i: 1\leq i<N).$$
 To set expectations, it is worth comparing with the size of the analogous quantity for some random sequences, when measured on the scale of the mean spacing between the levels in the sequence, which in our case is constant (equal to  $4\sqrt{\alpha}/\pi$).  
For a Poissonian sequence of $N$ uncorrelated levels with unit mean spacing, the smallest gap is almost surely of size  $\approx 1/N$ \cite{Levy}. In comparison, the smallest gap between the eigenphases of a random $N\times N$ unitary matrix is, on the scale of the mean spacing,  
 almost surely of size $ \approx N^{-1/3}$ \cite{Vinson, BB}, in particular much larger than the Poisson case. The same behaviour persists for the eigenvalues of random $N\times N$ Hermitian matrices (the Gaussian Unitary Ensemble) \cite{Vinson, BB}.  For the Gaussian Orthogonal Ensemble of random symmetric matrices, is is expected (though as of now not proved) that the minimal  gap is of size $N^{-1/2}$. 
We note that the local statistics of the eigenvalues of the Laplacian for generic chaotic systems, such as non-arithmetic surfaces of negative curvature, are expected to follow the Gaussian Orthogonal Ensemble \cite{BGS}, while the local statistics of the zeros of the Riemann zeta function are expected to follow the Gaussian Unitary Ensemble \cite{Montgomery, RS}.

  \subsection{Order of growth of $\delta^{(\alpha)}_{\min}(N)$}
Returning to our rectangular billiard,  it is not hard to obtain lower bounds for $\delta_{\min}^{(\alpha)}(N)$, see \S~\ref{sec:def of diophantine}. 
In the case of   quadratic irrationalities,  the gap function cannot shrink faster than $1/N$:  for each quadratic irrationality $\alpha$, there is some $c(\alpha)>0$ so that 
\begin{equation}\label{lower bd quad irr}
  \delta_{\min}^{(\alpha)}(N) \geq \frac{c(\alpha)}{N} \;.
\end{equation}
More generally, both for {\em algebraic} irrationalities and for almost every $\alpha$ (in the measure theoretic sense)  the same argument shows 
\begin{equation}\label{lower2}
  \delta_{\min}^{(\alpha)}(N) \gg  1/N^{1+\varepsilon}
\end{equation}
for any $\varepsilon > 0$, see Proposition \ref{1/N} below. Both \eqref{lower bd quad irr} and \eqref{lower2}   depend on general results in diophantine approximation. 

In \eqref{lower2} and elsewhere in the paper, we use Vinogradov's notation $f(N)\ll g(N)$ to mean that there are $c>0$ and $N_0\geq 1$ so that $|f(N)|\leq c |g(N)|$ for all $N>N_0$; and the notation $f(N)\asymp g(N)$ to mean both $f(N)\ll g(N)$ and $g(N)\ll f(N)$. Implied constants may always depend on $\alpha$ and $\varepsilon$ where applicable.

Much more work needs to be done to obtain good upper bounds for $\delta_{\min}^{(\alpha)}(N)$, i.e.\ to explicitly construct small gaps. 

 We show in Proposition \ref{prop:allalpha} below  that for any irrational  $\alpha$, we have 
\begin{equation} \label{trivial}
\delta_{\min}^{(\alpha)}(N) \ll  N^{-1/2}
\end{equation}
 for all $N$. By the same argument, we can also display $\alpha$ where $\delta_{\min}^{(\alpha)}(N) \ll N^{-A}$ for any $A > 0$   by taking $\alpha$ to be suitable Liouville numbers. However these form a measure zero set and are atypical.  %For typical (in the measure-theoretic sense) $\alpha$, the same simple argument giving \eqref{trivial} can be used to show that for almost all $\alpha$, and \emph{all} $N\geq1$, 
%\begin{equation}\label{simple argument}
%\delta_{\min}^{(\alpha)}(N) \ll  N^{-1/2+\varepsilon}
%\end{equation}
%for all $\varepsilon>0$ (this will be improved below).

For certain quadratic irrationalities we show that  the minimal gap  can be almost as small as $1/N$: 
% which is consistent with Poisson statistics. 
\begin{theorem}\label{thm:main}
If  the squared aspect ratio is a quadratic irrationality of the form $\alpha =\sqrt{r}$, with $r$ rational, then 
\begin{equation*}%\label{quadratic}
 \delta_{\min}^{(\alpha)}(N) \ll \frac 1{N^{1-\varepsilon}}
 \end{equation*}
for every $\varepsilon > 0$ and all $N$. 
\end{theorem}
We can also deal with other quadratic irrationalities, such as the golden mean. We refer to Section \ref{other} for more general results. In particular, we show in this section that there exist quadratic irrationalities $\alpha$ such that the  stronger result 
\begin{equation}\label{exact}
\delta_{\min}^{(\alpha)}(N) \ll 1/N
\end{equation}
holds for all $N$. An explicit example is the square of the golden mean 
$\alpha  =  (3 + \sqrt{5})/2.$

Moving away from quadratic irrationalities, where our results are deterministic, 
we turn to generic in measure $\alpha$.  
%If one is willing to assume the Riemann hypothesis, then  we show 
%the same bound  holds for \emph{all} $N$.
\begin{theorem}\label{thm:zeta} 
%Assume the Riemann hypothesis. 
For almost all $\alpha > 0$ (in the sense of Lebesgue measure) we have  
\begin{equation}\label{conditional}
 \delta^{(\alpha)}_{\min}(N) \ll  \frac{1}{N^{1-\varepsilon}}
 \end{equation}
for any $\varepsilon > 0$ and all $N$. 
\end{theorem}
%To prove \eqref{conditional}, one needs significantly less than the Riemann hypothesis; what is really needed is an essentially optimal bound for the eighth moment of the Riemann zeta function on the critical line:
%\begin{equation}\label{eighth}
%\int_{-T}^T |\zeta(1/2 + it)|^8 dt \ll  T^{1+\varepsilon}
%\end{equation}
%for all $\varepsilon > 0$.  For a precise conjecture for the asymptotic size of the eighth moment, see \cite{Keating Snaith}, \cite{Conrey Gonek}.

%Turning to unconditional results, we are able to show that for all $N$   we still have small gaps:
%\begin{theorem}\label{thm 4over7}
%For almost all $\alpha$,  any $\varepsilon > 0$ and all $N$, 
%\begin{equation}\label{unconditional}
% \delta^{(\alpha)}_{\min}(N) \ll \frac{1}{N^{3/4-\varepsilon}} \;.
% \end{equation}
%\end{theorem}
%Note that this is an improvement on the exponent $1/2$ in the simple result \eqref{simple argument}. 
%We can also produce gaps as small as $1/N^{1-\varepsilon}$, unconditionally, but only for infinitely many $N$, not necessarily for all $N$:
%\begin{theorem}\label{thm:Aemin} 
% For almost all $\alpha > 0$ (in the sense of Lebesgue measure) we have  
%$$ \delta^{(\alpha)}_{\min}(N) \ll  \frac{1}{N^{1-\varepsilon}} \quad\quad \text{infinitely often}$$
%for any $\varepsilon > 0$. 
%\end{theorem}

We summarize the preceding results by stating that the order of growth of $\delta_{\min}^{(\alpha)}(N) \approx 1/N$ is consistent with Poisson statistics for certain special and also generic in measure $\alpha$. 
However, as we now explain, finer details of Poisson statistics are always violated. 

 \subsection{Deviations from Poisson statistics}
  Given a sequence of   points, 
let $ \delta_{\min, k}(N)$  be the $k$-th smallest gap  ($k\geq 1$) among the first $N$ points in the sequence,  
so that in particular  $\delta_{\min, 1}(N)  =  \delta_{\min}(N)$. 
For a Poisson sequence with unit mean spacing (by which we mean $N$ points picked independently and uniformly in $[0,N]$), Devroye \cite{Devroye} showed that for any fixed $k \geq 1$ and any sequence $\{u_n\}$ of positive numbers such that $u_n/n^2$ is decreasing we have 
\begin{equation}\label{dev}
\text{Prob}\Big(N \delta_{\min, k}(N) \leq u_N \text{ infinitely often}\Big) = \begin{cases} 1, & \sum_n u_n^k/n = \infty,\\  \\0, & \sum_n u_n^k/n < \infty. \end{cases}
\end{equation}
Choosing for instance $u_n = 1/\log n$ for $k=1$, one has 
\begin{equation}\label{dev1}
\delta_{\min}(N) \leq \frac{1}{N \log N} \quad\quad \text{infinitely often}
\end{equation}
almost surely, while choosing $u_n = 1/(\log n)^{2/3}$ for $k = 2$ one has 
\begin{equation}\label{dev1a}
\delta_{\min, 2}(N) \geq \frac{1}{N (\log N)^{2/3}} \quad\quad \text{for all sufficiently large $N$}
\end{equation}
almost surely. 
Similarly, it is shown in \cite[Theorem 4.2]{Devroye} that
\begin{equation}\label{dev2}
\delta_{\min}(N) \geq \frac{\log\log N}{N } \quad\quad \text{infinitely often}
\end{equation}
almost surely, but by \cite[Theorem 4.1]{Devroye} we have
\begin{equation}\label{dev3}
 \text{Prob}\left( \delta_{\min}(N) \geq \frac{(\log\log N)^2}{N} \text{ infinitely often}\right)  = 0.
\end{equation} 

For our sequence $\{\alpha m^2+n^2\}$, we infer from \eqref{lower bd quad irr} that in the case of quadratic irrationalities %we have $\delta_{\min}^{(\alpha)}(N)\geq c(\alpha)/N$  for all $N$, so 
 \eqref{dev1} is violated. 
The following result  shows that \eqref{dev3} is violated for \emph{almost all} $\alpha$: 
\begin{theorem}\label{thm:mult table}
For almost all $\alpha>0$ (in the sense of Lebesgue measure) we have  
\begin{equation}\label{mult table}
 \delta^{(\alpha)}_{\min}(N)  \gg \frac{(\log N)^c}{N}  \quad\quad \text{infinitely often}
\end{equation}
where  
%\begin{equation}\label{c}
$c = 1 - \frac{\log (e \log 2)}{\log 2} = 0.086\ldots$. 
%\end{equation}
\end{theorem}

In fact, for \emph{all} $\alpha$ we show 
\begin{theorem}\label{thm:nonpoisson}
 For any $\alpha > 0$, at least one of the conditions   \eqref{dev1} or  \eqref{dev1a} is violated.  
\end{theorem}

%Indeed, it turns out that either there are no small gaps (violating \eqref{dev1}), or else if there is one than there are several of them, violating \eqref{dev1a}. The reason is that the sequence of eigenvalues is preserved under multiplication by a perfect square, hence if there is one small gap than there are several. 

It is also  of interest to study  the distribution of the \emph{largest} gap. 
One does expect arbitrarily large gaps, and it is a challenging problem to prove this for Diophantine $\alpha$. 
% but if we believe in the Poisson distribution in this situation,  one cannot expect more than gap of size $\log N$ %(in the case of unit mean spacing), rather than power growth.  

\subsection{About the proofs}\label{proofs}
The proofs draw from a  variety of methods. 
We show in Section \ref{sec: strategy}  (see Lemma \ref{general}) that the size of $\delta_{\min}^{(\alpha)}(N)$ depends on the existence of good 
rational approximants $p/q$ to $\alpha$, where both $p$ and $q$ are \emph{evenly divisible}, by which we mean   integers $n$ having a divisor $d\mid n$ roughly of size square root:  $$\min(d,n/d)\gg n^{1/2-\varepsilon}$$  
for any $\varepsilon > 0$.  
 We will see that the concept of evenly divisible numbers comes up naturally in the context of finding small gaps, although we have not seen it in other number theoretical applications.   

%To find   approximants $p/q$ with $p, q$ evenly divisible for a set of full measure of $\alpha$, we use work of Erd\H{o}s   \cite{Erdos} towards the Duffin-Schaeffer conjecture; this will prove Theorem \ref{thm:Aemin}. 

To find such approximants for certain quadratic irrationalities, for instance  $\alpha=\sqrt{D}$ as in Theorem \ref{thm:main}, for integer $D>1$ not a perfect square, we use the theory of Pell's equation to show 
that there are many approximants $p_n/q_n$ for which both of the sequences $\{p_n\}$ and $\{q_n\}$ satisfy a ``strong divisibility" condition of the form $$\gcd(a_m,a_n) = a_{\gcd(m,n)},\quad m,n\;{\rm odd}.
$$ 
This  condition can be used to produce ``good" divisors. 

Theorem \ref{thm:zeta} uses a second moment approach to obtain a result valid for almost all $\alpha$. The corresponding counting problem that produces evenly divisible approximants is analyzed by exponential sums, and becomes naturally a problem in 4 variables, so that the second moment produces an eighth moment of the Riemann zeta-function. In absence of the Lindel\"of hypothesis, we introduce artificially a bilinear structure, separating the 4 variables into 4 short ones and 4 long ones;  we obtain an unconditional saving on the short variables using strong bounds for the Riemann zeta function $\zeta(s)$ near the line $\Re(s)=1$ based on Vinogradov's method, and 
handle the contribution of the long variables using a mean-value theorem. 
%, requiring  
%This can be expressed in terms of the eighth moment of the Riemann zeta function. To obtain unconditional results, We need 
%nontrivial bounds for exponential sums
%$$\sum_{N < n < 2N} n^{-i   x}$$
%for $N^3 \leq x \leq N^4$, and uses bounds for the Riemann zeta function $\zeta(s)$ near the line $\Re(s)=1$. 
%\marginpar{OK?} 
%Here we use   Jutila's \cite{Ju} by now classical innovations on large values of Dirichlet polynomials, combined with the theory of exponent pairs. 
The general scheme of this method has already found further applications in connection with the Oppenheim conjecture for ternary quadratic forms \cite{Bourgain}. 
%bounds for the measure of import a recent estimate of Heath-Brown \cite{HB}, which builds on work of  Wooley \cite{Wooley} and of Bourgain, Demeter and Guth \cite{BDG} on Vinogradov's mean value theorem. 

To prove the lower bound in Theorem~\ref{thm:mult table}, we invoke Ford's  quantitative  version \cite{Fo} of the result first proved by Erd\H{o}s \cite{Er1} that a multiplication table of side length $X$ contains $o(X^2)$ different  entries, which gives restrictions on the arithmetic properties of approximants.

\subsection*{Acknowledgement} The authors would like to thank Peter Sarnak for useful comments, and the referee for a very careful reading of the manuscript.

\section{Some general results} 

 \subsection{Lower bounds}\label{sec:def of diophantine}

An irrational $\alpha$  is {\em badly approximable} if  for all integers $(p,q)$ with $q\geq 1$  we have 
\begin{equation}\label{21}
 |q\alpha-p|\gg \frac 1q \;.
\end{equation}
It is (strongly) Diophantine if  we have the weaker inequality 
\begin{equation}\label{22}
|q\alpha-p|\gg  \frac 1{q^{1+\varepsilon}} \quad \text{for all} \quad  \varepsilon>0 \;.
\end{equation}

We recall \cite{Khinchine} that  $\alpha$ being badly approximable is equivalent to having bounded partial quotients in the continued fraction expansion of $\alpha$. 
% Khinchine chapter II theorem 23
Thus quadratic irrationalities are badly approximable. 
The set of badly approximable reals has measure zero. However   the set of (strongly) Diophantine numbers has full measure. %Khinchine chapter III theorem 29 
Roth's theorem says that all algebraic irrationalities are (strongly) Diophantine.  For a full measure set of $\alpha$, one in fact has a stronger lower bound \cite{Khinchine}: For every $\varepsilon>0$, we have %there exists $Q = Q(\alpha, \varepsilon)$ such that
\begin{equation}\label{23}
|q\alpha-p|\gg \frac 1{q(\log q)^{1+\varepsilon}} . %\quad  \text{for all} \quad q>Q.
\end{equation}
for all $q \geq 2$. 
%Khinchine chapter III theorem 32

\begin{proposition}\label{1/N} Let $\alpha > 0$.

i) Suppose $\alpha\in \R\backslash \Q$ is badly approximable. Then for all $N$ we have 
$$ \delta_{\min}^{(\alpha)}(N)\gg \frac 1N\;.
$$

ii) If $\alpha \in \R\backslash \Q$ is (strongly) Diophantine, then for all $\varepsilon > 0$ and all  $N$ we have 
$$ \delta_{\min}^{(\alpha)}(N)\gg \frac 1{N^{1+\varepsilon}}\;.
$$

iii) For Lebesgue almost all $\alpha$, for all $\varepsilon>0$ and all $N$ we have 
$$
 \delta_{\min}^{(\alpha)}(N)\gg \frac 1{N (\log N)^{1+\varepsilon}}\;.
$$
\end{proposition}
\begin{proof}
Indeed if $\alpha$ is badly approximable  then for any two distinct eigenvalues $\lambda:=\alpha m^2+n^2$ and $\lambda':=\alpha m'^2+n'^2$ with $\max(\lambda, \lambda') \leq N$ we obtain 
$$ 
|\lambda-\lambda'| = |(m^2-m'^2)\alpha - (n'^2-n^2)|\gg \frac 1{|m^2-m'^2|}\gg 
 \frac 1{\max(\lambda,\lambda')} \geq \frac{1}{N}
$$
%and since by Weyl's law, the number of eigenvalues up to $\max(\lambda,\lambda')$ is $N\approx \max(\lambda,\lambda')$, we obtain the lower bound $ \delta_{\min}^{(\alpha)}(N)\gg 1/N$. 
using \eqref{21}. 
The same argument with \eqref{22} and \eqref{23} in place of \eqref{21} proves ii) and iii). %gives $ \delta_{\min}^{(\alpha)}(N) \gg 1/N^{1-o(1)}$ for (strongly) Diophantine numbers.
\end{proof}

\subsection{A general upper bound} % on $\delta_{\min}^{(\alpha)}(N)$} 

\begin{proposition}\label{prop:allalpha} 
For any  irrational  $\alpha > 0$, we have
$$ \delta_{\min}^{(\alpha)}(N)\ll N^{-1/2} $$
for all $N$. 
%Moreover, for almost all $\alpha$ (in the sense of measure), for \emph{all} $N\gg 1$ we have 
%$$ \delta_{\min}^{(\alpha)}(N)\ll N^{-1/2+\varepsilon } $$
%for all $\varepsilon>0$. 
\end{proposition}
\begin{proof} Let $Q \geq 1$ be sufficiently large.  By Dirichlet's approximation theorem there exist integers $a \in \Bbb{Z}$,   $1 \leq q \leq Q$ such that  $0<|a - q\alpha| \leq 1/Q$, and since $\alpha > 0$ we must have $a \geq 1$. With $m = 2q+1$, $m' = 2q - 1$, $n = 2a-1$, $n' = 2a + 1$ we have $1 \leq m, m', n, n' \ll Q$ and 
$$|\alpha m^2 + n^2 - (\alpha {m'}^2 - {n'}^2)| = 8 |\alpha q - a| \leq \frac{8}{Q}.$$
Choosing $Q$ to be of exact order $N^{1/2}$ gives the desired bound. 
\end{proof}

\section{The general strategy}\label{sec: strategy} 
From now on, we deal with getting a bound of the form $\delta_{\min}^{(\alpha)}(N)\ll N^{-1+\varepsilon}$. 
We will frequently use the relation $\lambda_i \asymp i$ for $i \geq 1$. 

We recall the notion of ``evenly divisible'', introduced in Section \ref{proofs}.  
\begin{definition}
We call an integer  $n$ is evenly divisible if there is a divisor $d\mid n$,  such that $\min(d,n/d)\gg n^{1/2-\varepsilon}$ for all $\varepsilon > 0$.   We call $n$ strongly evenly divisible if there is a divisor $d\mid n$, such that $\min(d,n/d)\gg n^{1/2}$.
\end{definition}
%\marginpar{Can we find a better name?}  
So primes are not evenly divisible, but perfect squares are, even   strongly so.  
Suppose we have found a good rational approximation %of $\alpha$  
\begin{equation}\label{approx}
|\alpha q-p| \ll \frac 1{q}  
\end{equation}
with $p$, $q$ both evenly divisible, say $d\mid q$, $q^{1/2-\varepsilon}\ll d\leq \sqrt{q}$, and $e\mid p$, $p^{1/2-\varepsilon}\ll e\leq \sqrt{p}$ (note that $p\asymp q$ since $p/q $ is an approximation to $\alpha$). It is useful to observe that we may assume without loss of generality that neither $p$ nor $q$ is a perfect square. Indeed, at least one of the pairs $(p, q)$, $(2p, 2q)$, $(3p, 3q)$ contains two non-squares, and so we can simply replace $(p, q)$ with $(2p, 2q)$ or $(3p, 3q)$ in \eqref{approx} if necessary.

Now find $m>m'\geq 1$, $n'>n\geq 1$  solving 
$$ m-m'=2d,\quad m+m' = 2\frac{q}{d},\quad n-n'=2e,\quad n+n' = 2\frac{p}{e}, $$
namely $$  m=\frac {q}{d}+d  , \quad  m'= \frac {q}{d}-d , \quad 
n= \frac{p}e-e,\quad n'= \frac{p}e+e\;.
$$
Notice that all variables are non-zero by our assumption that neither $p$ nor $q$ is a perfect square.  
Clearly  
$$ 
m^2-m'^2  =4q    ,\quad n'^2-n^2 =4 p  
 $$
and moreover  by our assumptions on the size of $d$ and $e$, we have 
$$ q^{1/2} \ll m, n'  \quad \text{and} \quad  m,m',n,n'\ll q^{1/2  +\varepsilon} \;.
$$

Hence the corresponding eigenvalues  
$$ \lambda:=\alpha m^2+n^2, \quad \lambda':=\alpha m'^2+n'^2 $$
 satisfy (maybe with a different value of $\varepsilon$)  
$$ q \ll \lambda,\lambda'\ll q^{1+\varepsilon} 
$$
 and give a gap in the spectrum   of size at most
$$
|\lambda-\lambda'| = |\alpha (m^2-m'^2) - (n'^2-n^2)| =4| \alpha q -p | \ll \frac 1{q} \ll \frac{1}{\max(\lambda, \lambda')^{1-\varepsilon}},
$$
%Let $N$ be the number of eigenvalues up to $\max(\lambda,\lambda')$. Then 
%$$N\approx \max(\lambda,\lambda')\ll q^{1+o(1)} $$
%and so
%$$
%|\lambda-\lambda'| \ll\frac 1{q} \ll  \frac 1{N^{1-o(1)}}
%$$
%which gives the lower bound 
where we used \eqref{approx} in the penultimate step. We conclude
$$
\delta_{\min}^{(\alpha)}(N)\ll \frac 1{N^{1-\varepsilon}}\;
$$
for $N \asymp \max(\lambda, \lambda')$. 
This argument shows  the following:
\begin{lemma}\label{general}
If $\alpha > 0$ has infinitely many good rational approximations $p_n/q_n$ with $q_1 < q_2 < \ldots $ as in \eqref{approx} with 
%Thus we found rational approximations $p/q$ of $\alpha$ with 
both $p$ and $q$ evenly divisible (resp.\ strongly evenly divisible),  then $\delta_{\min}^{(\alpha)}(N)\ll N^{-1+\varepsilon}$ for all $\varepsilon > 0$  (resp.\ $\delta_{\min}^{(\alpha)}(N)\ll N^{-1}$) infinitely often.\\
If in addition $q_n \geq c q_{n+1}$, for some constant $c > 0$ (possibly depending on $\alpha$, but not on $n$), then these inequalities hold   for all $N$. 
\end{lemma}
For later purposes we record the following variation. If we replace \eqref{approx} with the weaker condition
\begin{equation}\label{approx1}
\Bigr|\alpha - \frac{p}{q}\Bigl| \ll \frac{1}{T}
\end{equation}
for some $T \leq q^2$, we obtain the following:
\begin{lemma}\label{general1}
If $\alpha > 0$ has infinitely many good rational approximations $p_n/q_n$ with $q_1 < q_2 < \ldots $ as in \eqref{approx1} with both $p$ and $q$ evenly divisible and $q_n \geq c q_{n+1}$ for some constant $c > 0$ and  all $n \geq 1$, then 
 $$\delta_{\min}^{(\alpha)}(N) \ll N^{1+\varepsilon} T^{-1}$$ for all $N$ and all $\varepsilon > 0$. \end{lemma}

\section{Interlude: Strong divisibility sequences and Chebyshev polynomials}

A sequence of integers $\{a_n\}$ is a {\em divisibility sequence} if $m\mid n$ implies that  $a_m\mid a_n$. It is a {\em strong} divisibility sequence if 
$$ \gcd(a_m, a_n) = a_{\gcd(m,n)}.$$
A classical example is the sequence of Fibonacci numbers (see \cite[Section 1.2.8]{Knuth}), and it is known that second order recurrence sequences with constant coefficients of the form 
\begin{equation}\label{strong}
a_{n+1} = b a_n + d a_{n-1}, \quad (b, d) = 1, \quad a_0 = 0, \quad a_1 = 1
\end{equation}
satisfy this property, see e.g.\  \cite[Proposition 2.2]{HS}.

 One can generate families of such sequences with Chebyshev polynomials. We recall that the Chebyshev polynomials of the first and second kind $T_n$ and $U_n$ are defined as (see e.g. \cite{Rivlin})
$$T_n(x) = \frac{1}{2}\left((x + \sqrt{x^2-1})^n + (x - \sqrt{x^2-1})^n\right)$$
and
$$U_n(x) = \frac{ (x + \sqrt{x^2-1})^{n+1} - (x - \sqrt{x^2-1})^{n+1}}{2 \sqrt{x^2 -1}}.$$
They satisfy the second order recurrence relation
\begin{equation}\label{rec}
T_{n+1}(x) = 2x T_{n}(x) - T_{n-1}(x), \quad U_{n+1}(x) = 2x U_{n}(x) - U_{n-1}(x),
\end{equation}
and they are solutions of a polynomial Pell equation
\begin{equation}\label{pell}
T_n(x)^2 - (x^2 - 1)U_{n-1}(x)^2 = 1.
\end{equation}
Also useful is the formula
\begin{equation}\label{formula}
  T_{n+m}(x) = 2 T_n(x) T_m(x) - T_{n-m}(x), \quad n \geq m \geq 0,
\end{equation}
which can be easily verified from the definition of $T_n$. 
One checks   by induction using  \eqref{rec} that
\begin{equation}\label{integral}
  U_n(x/2), \, 2T_n(x/2) \in \Bbb{Z} \quad \text{for } x \in \Bbb{Z}.
\end{equation}
Any half-integral specialization of shifted Chebyshev polynomials of the second kind forms a strong divisibility sequence:
\begin{equation}\label{second}
  (U_{n-1}(x/2), U_{m-1}(x/2)) = U_{(n, m)-1}(x/2)
\end{equation}
for all $n, m, x \in \Bbb{N}$. This follows, for instance, from noting that  the sequence $a_n = U_{n-1}(x/2)$ satisfies \eqref{strong} with $d = -1$, $b = x$, see also  \cite{RTW}. 

A little less known is a slightly weaker corresponding statement for Chebyshev polynomials of the first kind: we have 
\begin{equation}\label{first}
  (2T_n(x/2), 2T_m(x/2)) = 2T_{(n, m)}(x/2)
\end{equation}
for all $x \in \Bbb{N}$ and all \emph{odd} positive integers $n, m$. A variation of this is proved in \cite[Theorem 2]{RTW}, but for convenience we give a proof of this fact: 

Let  $x \in \Bbb{Z}$, and let $a_n = 2T_n(x/2)$.  We write $y = \frac{1}{2}(x + \sqrt{x^2 - 4})$, so that $2T_n(x/2) = y^n + y^{-n}$. Clearly $y$ is a quadratic algebraic integer of norm 1, since it is the root of a monic integral quadratic polynomial. Let $m$ be odd. Then clearly 
$$2T_{nm}(x/2) = 2T_n(x/2) \sum_{j=0}^{m-1} (-1)^j y^{n(2j-(m-1))}, $$
and by basic Galois theory, the second factor is rational and an algebraic integer, hence integral. This shows $a_{n} \mid a_{nm}$ for every odd $m$. Next suppose that $n, m$ are both odd. We know already $a_{(n, m)} \mid (a_n, a_m)$, and we want to show equality here. Write
$$2(n, m) = rn - sm$$ 
with \emph{odd} positive integers $r, s$. Then
$$(a_n, a_m) \mid (a_{rn}, a_{sm}) = (a_{sm+2(n, m)}, a_{sm}).$$
Applying \eqref{formula} recursively with $(sm - 2j(n, m), (n, m))$, $j = 0, 1, \ldots$, in place of $(n, m)$ we see that
$$(a_{sm + 2(n, m)}, a_{sm}) \mid (a_{sm}, a_{sm - 2(n, m)}) \mid \cdots \mid (a_{3(n, m)}, a_{(n, m)}) = a_{(n, m)},$$
as desired.

\section{Rational approximants of $\sqrt{D}$}\label{sec:SDS}

In this section we prove Theorem \ref{thm:main}. Let $\alpha = \sqrt{r} \not\in \Bbb{Q}$, $r \in \Bbb{Q}_{>0}$, be given. By Lemma \ref{general} it suffices to find a sequence $p_n/q_n$, $q_1 < q_2 < \ldots$,  of approximations $|\alpha - p_n/q_n | \ll 1/q_n^{2}$ such that $p_n$ and $q_n$ are simultaneously evenly divisible, and $q_n \gg q_{n+1}$.  To simplify things, we observe   that we can restrict  $r$ to be an integer divisible by 4, say $r = 4D$ with $D \in \Bbb{N}$ not a perfect square, since fixed rational factors can be distributed among the $p_n$ and $q_n$ without changing the notion of evenly divisible, nor the quality of  the approximation, nor the  inequality $q_n \gg q_{n+1}$. 

By the theory of Pell's equation there exists a non-trivial solution $(x, y) \in \Bbb{N} \times \Bbb{N}$ to the diophantine equation
$$x^2 - Dy^2 = 1.$$
Consider the sequences
$$x_n := T_n(x), \quad  y_n := y U_{n-1}(x).$$
By \eqref{pell}, these are also (obviously pairwise different) solutions of the Pell equation, since
$$x_n^2 - Dy_n^2 = T_n(x)^2 - Dy^2 U_{n-1}(x)^2  = 1.$$
Therefore
$$\left|\sqrt{4D} - \frac{2x_n}{y_n} \right| \leq \frac{2}{\sqrt{D} y_n^2} \ll \frac{1}{y_n^2}.$$
It is clear from the definition of the Chebyshev polynomials that
\begin{equation}\label{growth}
\log x_n, \, \log y_n = n \log (x + \sqrt{x^2 - 1}) + O(1)
\end{equation}
for $n \rightarrow \infty. $

Now given $0 < \varepsilon < 1/2$, we can find  distinct odd primes $2<\ell_1<\ldots<\ell_J$ coprime to $y$ so that 
\begin{equation}
\frac 12-\varepsilon <1-\prod_{j=1}^J \left(1-\frac 1{\ell_j}\right) <\frac 12. 
\end{equation}
This is because  $\{ 1/\ell: \ell\;{\rm prime}\}$ is a zero sequence whose sum is divergent (this is a form of the Riemann rearrangement theorem). 
%\marginpar{This is some form of the Riemann rearrangement theorem (or at least that's how I would prove it). Is there a name that goes with it?} 
For instance, take 
$ \ell_1 = 3$, $\ell_2 = 5$, $\ell_3 = 17$, $\ell_4 = 257$ with 
$$1-\prod_{j=1}^4 \left(1-\frac 1{\ell_j}\right) \approx 0.499992.$$ 
From now on we consider  indices $n$ of the form 
\begin{equation}\label{n}
n:=\ell_1\cdot \ell_2\cdot \ldots \cdot \ell_J\cdot P, 
\end{equation}
where  $P$ is any odd large positive integer coprime to $\ell_1 \cdot \ldots \cdot \ell_J$ 
(note that such $n$'s  are odd). Put $p_n = 2x_n$, $q_n = y_n$. By \eqref{second} and \eqref{first}, $q_{n/\ell_j}\mid q_n$ and $p_{n/\ell_j}\mid p_n$ for each $j$.    
Therefore setting   
$$
Q:=\lcm (q_{n/\ell_1},\ldots, q_{n/\ell_J}), \quad P:=\lcm (p_{n/\ell_1},\ldots, p_{n/\ell_J}),
$$
we get divisors  $Q\mid q_n$ and $P\mid p_n$.

%We have $\log p_n\sim n\log\alpha$  (!!! justify!!!) and hence  $\log E/\log p_k \sim \frac 12$, so that for $k\gg 1$ we find that 
%$e:=\min(E,p_k /E)$ satisfies our requirements. 

%We   want to argue that $d:=\min(D,q_k/D)$ satisfies our requirements for ``evenly divisible'', i.e.\ %which would follow from knowing that $\log D/\log q_k\sim 1/2$. We use the formula
We want to argue that $Q$ is a divisor of $q_n$ roughly of square root size, so that $q_n$ is evenly  divisible. Precisely, we will show that
\begin{equation}\label{1/2}
 \left(\frac{1}{2} -\varepsilon \right) \log q_n + O(1) \leq \log Q \leq   \left(\frac{1}{2} +\varepsilon \right) \log q_n + O(1).
\end{equation}
To this end we recall the inclusion-exclusion formula for the least common multiple
\begin{displaymath}
\begin{split}
\lcm(a_1, \ldots, a_r)&  =   \prod_{\substack{S \subseteq \{1, \ldots J\}\\ |S| \geq 1}} \gcd(\{a_j \mid j \in S\})^{(-1)^{|S| - 1}},\\
& = \prod_{1 \leq j \leq J}  a_j \prod_{1 \leq i < j \leq J} \gcd(a_i, a_j)^{-1}   \prod_{1 \leq i < j  < k\leq J} \gcd(a_i, a_j, a_k ) \ldots,
\end{split}
\end{displaymath}
so that by \eqref{second} we obtain
\begin{displaymath}
\begin{split}
  \log Q & = \sum_{1 \leq j \leq J} \log q_{n/\ell_j} -   \sum_{1 \leq i < j \leq J} \log q_{\gcd(n/\ell_i, n/\ell_j)}  + \ldots\\
  & = \sum_{1 \leq j \leq J} \log q_{n/\ell_j} -   \sum_{1 \leq i < j \leq J} \log q_{ n/(\ell_i \ell_j)}  + \ldots,
  \end{split}
\end{displaymath}
where in the second step we used that $\ell_1,    \ldots , \ell_J$ are pairwise coprime divisors of $n$   and hence $$\gcd(\{n/\ell_j \mid j \in S\}) = \frac{n}{\prod_{j\in S} \ell_j }.$$
 By \eqref{growth} this  equals 
\begin{displaymath}
\begin{split}
&\log \left(x + \sqrt{x^2 - 1}\right) \Bigl( \sum_{1 \leq j \leq J}\frac{n}{\ell_j} - \sum_{ 1 \leq i < j \leq J}  \frac{n}{\ell_i \ell_j} +\ldots   \Bigr) + O(1)\\
= &  \log \left(x + \sqrt{x^2 - 1}\right)  n \Biggl( 1-\prod_{j=1}^J \Bigl(1-\frac 1{\ell_j}\Bigr) \Biggr) + O(1)\\
= &\log q_n \Biggl( 1-\prod_{j=1}^J \Bigl(1-\frac 1{\ell_j}\Bigr) \Biggr) + O(1),
\end{split}
\end{displaymath}
which gives \eqref{1/2}.

Likewise, $\log P \sim\frac{1}{2} \log p_n$, so that $p_n$ is evenly divisible. 
%\end{proof}

Finally we observe that the  admissible indices $n$ as in \eqref{n} is an integer sequence with bounded gaps (e.g.\ by $(\ell_1 \cdot \ldots \cdot \ell_J)^2$), and it follows directly from the definition of $q_n = y U_{n-1}(x)$ that  $q_n \gg q_{n+1}$. This completes the proof of Theorem \ref{thm:main}.  \hfill $\square$

\section{Some other quadratic irrationalities}\label{other}
We can leverage our results about irrationalities of the form $\sqrt{D}$ to obtain the same result on $\delta_{\min}^{(\alpha)}(N)$ for other quadratic irrationalities.  
\begin{theorem}\label{moregeneral}
For all positive  real quadratic irrationalities of the form
\begin{equation}\label{alpha}
\alpha = \alpha(x;a,b,\varepsilon,r) =r\cdot  \left(\frac{x+\sqrt{x^2 + 4\varepsilon}}{2}\right)^a\cdot \left(\sqrt{x^2+4\varepsilon} \right)^b
\end{equation}
with 
$$a\in \Z, \quad b=0,1,\quad  x\in \Z\backslash \{0\}, \quad \varepsilon =\pm 1, \quad r\in \Q^\times,
$$
we have $\delta^{(\alpha)}_{\min}(N) \ll N^{- 1 + \varepsilon}$ infinitely often.
\end{theorem}
In order to have $\alpha \in \Bbb{R} \setminus \Bbb{Q}$ we need $(a, b) \not= 0$,   and in addition $x \not\in \{0, \pm 1, \pm 2\}$ if $\varepsilon = -1$.  We can also assume $x > 0$, since $\alpha(-x, a, b, \varepsilon, r) = \alpha(x, -a, b, \varepsilon, (-1)^{\varepsilon a} r)$.

Note that we can display any irrationality of the form $\sqrt{D}$, with integral $D>1$ not a perfect square, as such $\alpha$: Indeed, let $z^2-Dw^2=1$ be a nontrivial solution to the corresponding Pell equation. Choosing $$r=1/(2w), \quad a=0, \quad b=1, \quad x=2z, \quad \varepsilon = -1$$ gives $\sqrt{D} = \alpha(2z,0, 1,-1,1/(2w))$. In particular, Theorem \ref{thm:main} is a special case of Theorem \ref{moregeneral}.  The golden ratio $(1+\sqrt{5})/2$ is obtained by taking $r=1$, $a=1$, $b=0$, $x=1$, $\varepsilon=1$. There are many other examples, but we do not know how to cover all quadratic irrationalities.

\begin{proof} Good rational approximants for $\alpha$ are obtained from the relations
\begin{equation}\label{cases}
\begin{split}
\alpha\Big(x,a,0,-1,\frac cd\Big) & = \frac{c U_{n+a}(x/2)}{dU_n(x/2)} + O\Big( \frac 1{U_n(x/2)^2}\Big),\\
\alpha\Big(x,a,0,+1,\frac cd\Big) &= \frac{c i^{-n-a}U_{n+a}(ix/2)}{d i^{-n}U_n(ix/2)} + O\Big( \frac 1{  |U_n(ix/2)|^2} \Big),\\
\alpha\Big(x,a,1,-1,\frac cd\Big) &= \frac{c 2 T_{n+a}( x/2)}{dU_{n-1}(x/2)} + O\Big( \frac 1{U_{n-1}(x/2)^2} \Big),\\
\alpha\Big(x,a,1,+1,\frac cd\Big)& = \frac{c 2i^{-n-a} T_{n+a}(ix/2)}{di^{1-n}U_{n-1}(ix/2)} + O\Big( \frac 1{|U_{n-1}(ix/2)|^2}\Big),
\end{split}
\end{equation}
which follows immediately from the definition of the Chebyshev polynomials. By the above remarks, 
this covers all $\alpha$ considered in Theorem \ref{other}. Notice that numerators and denominators are integers in each case. We now proceed similarly as in the previous section. We choose odd primes
$2<\ell_1<\ldots<\ell_r$, with $\ell_i=1\mod 4$,  so that
%\marginpar{I set the $\ell_i=1\mod 4$ and the $\ell'_j=3\mod 4$ to guarantee they are distinct and $L$, $L'$ coprime} 
\begin{displaymath}
\frac 12-\varepsilon <1-\prod_{j=1}^r \left(1-\frac 1{\ell_j}\right) <\frac 12. 
\end{displaymath}
and we choose another set of distinct odd primes %(different from all the previous ones) 
$2<\ell'_1<\ldots<\ell'_s$, with $\ell_j'=3\mod 4 $,   so that 
\begin{displaymath}
\frac 12-\varepsilon <1-\prod_{j=1}^s \left(1-\frac 1{\ell'_j}\right) <\frac 12. 
\end{displaymath}
Put
$$L = \prod_{j=1}^r \ell_j, \quad L' = \prod_{j=1}^r \ell_j'.$$
By construction $(L, L') = 1$. 
We put $n+a = Lm$; moreover, in the   first two cases of \eqref{cases} we put $n+1 = L'm'$, in the last two cases of \eqref{cases} we put $n = L'm'$ with $Lm$ odd and $(L, m) = (L', m') = 1$. Then by the argument of the previous section,  numerators and denominators of the approximations are evenly divisible. It remains to show that we can pick infinitely such pairs $(m, m')$. To this end we put $$m' = \mu'L' + 1, \quad m = 2\mu L + 1,$$  
so that $(L, m) = (L', m') = 1$ and $mL$ is odd, and the linear diophantine equation
$$Lm - L'm' = 2L^2 \mu  - \mu' (L')^2  + (L - L') = b$$
has, for any $b\in \Bbb{Z}$, infinitely many pairs of solutions $(\mu, \mu')$, since $(2L^2, (L')^2) = 1$. \\

In certain cases we can do a little better, and we conclude this section with a proof of \eqref{exact} for all $\alpha(x, a, 0, \pm 1, c/d)$ with $a$ even.  In this case we are dealing exclusively with Chebyshev  polynomials of the second kind, for which slightly better divisibility conditions hold. In particular, restricting the first two cases of \eqref{cases} to odd $n$ and assuming that $a$ is even, the indices in numerator and denominator are odd, and it follows from \eqref{second} that
$$U_{(n+a-1)/2}(x/2) \mid U_{n+a}(x/2), \quad U_{(n-1)/2}(x/2) \mid U_{n}(x/2),$$
so that every second approximant of $\alpha$ has numerator and denominator that are strongly evenly divisible.  \end{proof}

\section{Almost all $\alpha$, lower bound: Proof of Theorem~\ref{thm:mult table}  }

Without loss of generality we will prove Theorem~\ref{thm:mult table}  for almost all $\alpha \in [1, 2]$. Of course the same argument works for any other interval. For $N, X \geq 1$ and $q \in \Bbb{N}$ let
$$S_{X, N} = \{\alpha \in [1, 2] \mid \delta_{\min}^{(\alpha)}(N) \leq 1/X\}$$
and
$$S_X^{(q)} = \{ \alpha \in [1, 2] \mid \| \alpha q \| \leq 1/X\},$$
where as usual $\| . \|$ denotes the distance to the nearest integer. Clearly
$$\mu\big(S_X^{(q)}\big) = \frac{2}{X}.$$
Then $\alpha \in S_{X, N}$ implies that there exist $m, m', n, n' \ll N^{1/2}$ such that 
$$|\alpha(m^2 - m'^2) - (n'^2 - n^2)| \leq 1/X$$
and in particular there exist $u = m - m'$,  $v = m + m'$ with $u, v  \ll N^{1/2}$ such that $$\alpha  \in S_X^{(uv)}.$$
We conclude that
$$S_{X, N} \subseteq \bigcup_{u, v \leq C N^{1/2}} S_X^{(uv)}$$
for a suitable constant $C > 0$ (depending on $\alpha$). Note that the sets $S_X^{(uv)}$ are indexed by the integers which are products $u\cdot v$ with $u,v\leq CN^{1/2}$. These are just the distinct elements in a multiplication table of side length $CN^{1/2}$. 
%\marginpar{Added some explanation} 
 Erd\"os showed \cite{Er1} that a multiplication table of side length $X$ contains $o(X^2)$ different  entries. We now invoke Ford's  quantitative  version \cite[Corollary 3]{Fo}, which shows that the union is over $\ll N (\log N)^{-c}(\log\log N)^{-2/3}$ pairs with $c = 1 - \frac{\log (e \log 2)}{\log 2} = 0.086\ldots$.  
We obtain
$$\mu(S_{X, N}) \ll \frac{N}{(\log N)^{c} (\log\log N)^{2/3}} \cdot \frac{1}{X}.$$

Now let
$$S := \big\{\alpha \in [1, 2] \mid \delta_{\min}^{(\alpha)}(N) \leq (\log N)^{c}/N \text{ for all sufficiently  large } N \big\}.$$
Then
$$S = \liminf_{N \rightarrow \infty} S_{ N/(\log N)^{c}, N},$$
and since 
$$\mu(S_{N/(\log N)^{c}, N}) \ll \frac{1}{(\log\log N)^{2/3}} \rightarrow 0,$$
it is clear that $\mu(S) = 0$.

\section{Almost all $\alpha$: bounds for all $N$}
In this section we prove Theorem~\ref{thm:zeta}   for almost all $\alpha \in \mathcal{J}$ (without loss of generality), where   $\mathcal{J} \subseteq (0, \infty)$ is some fixed compact interval. In the following, all implied constants may depend on $\mathcal{J}$. 

 For $\alpha \in \mathcal{J}$,   real $M \geq 1$  and  $M^3 \leq T \leq M^4$,  let 
\begin{equation}\label{S}
\mathcal{S}(M, T, \alpha) := \#\Bigl\{ n_1, n_2, n_3, n_4  \asymp M : \Bigl|\frac{n_1n_2}{n_3n_4}  - \alpha \Bigr| \ll \frac{1}{T} \Bigr\}.
\end{equation}
We are interested in a lower bound for this quantity for almost all $\alpha$  and $T$ as large as possible in terms of $M$. We will  prove the following
\begin{proposition}\label{prop9.1}
For any $\eta > 0$ sufficiently small, we have
$S(M, M^{4 - \eta}, \alpha) \geq 1$ for all sufficiently large $M \geq M_0(\alpha)$, and all 
$\alpha \in \mathcal{J} \setminus \mathcal{T}_M$, 
where $\mathcal{T}_M$ is an exceptional set of measure 
$\mu(\mathcal{T}_M) \ll M^{-\rho}$ with $\rho > 0$ depending only on $\eta > 0$.
%
%We have 
%$S(M, M^{c}, \alpha) \geq 1$ for all sufficiently large $M$, and for any $c=4-\eta$ with $\eta>0$ sufficiently small  
%%a certain constant $3 < c < 4$ 
%and all  $\alpha \in \mathcal{J} \setminus \mathcal{T}_M$, where $\mathcal{T}_M$ is an exceptional set of measure $\mu(\mathcal{T}_M) %\ll M^{-\rho}$ for some $\rho > 0$ depending at most on $\eta$. 
%%Assuming the eighth moment bound  \eqref{eighth},   we can choose $c = 4 - \varepsilon$; unconditionally we can choose $c = 7/2 - \va%repsilon$. 
\end{proposition}
%The reader may be warned that in this section  the value of $\varepsilon$ may change from line to line, in particular, the $\varepsilon$'s in the statement of the proposition need not be the same.

Taking Proposition \ref{prop9.1} for granted, we specialize  $M = 2^{\nu}$, $\nu \in \Bbb{N}$,  so that 
$$\sum_{M= 2^{\nu}} \mu(\mathcal{T}_M) < \infty.$$
By the Borel-Cantelli lemma we conclude $S(M, M^{4 - \eta}, \alpha) \geq 1$ for almost all $\alpha$, all sufficiently large $M = 2^{\nu} \geq M_0(\alpha)$ and $\eta$ as in Proposition \ref{prop9.1}.  It follows from Lemma \ref{general1} that 
\begin{equation}\label{delta}
\delta^{(\alpha)}_{\min}(N) \ll N^{1-\frac{4 - \eta}{2}+\varepsilon}  = N^{-1 + \eta/2 + \varepsilon}
\end{equation}
for \emph{all} sufficiently large integers $N \geq N_0(\alpha)$. Since we allow the implied constant to depend on $\alpha$, \eqref{delta} holds in fact for all $N$, and the 
  bound   \eqref{conditional} in Theorem~\ref{thm:zeta} follows.
  
  The remainder of this section is devoted to the proof of Proposition \ref{prop9.1}. 
To  prepare for the upcoming Fourier analysis,  let $w_1, w_2$ be two non-negative smooth functions bounded by 1. We assume that   $w_1$ takes the value 1 on some sufficiently large interval $[a_1, b_1]$ with constants $0 < a_1 < b_1$ depending on $\mathcal{J}$  
and the value 0 outside $[\frac{1}{2} a_1, 2b_1]$, and that   $w_2$ takes the value 1 on   $[-1, 1]$   and the value 0 outside $[-2, 2]$.  Note  that the Fourier transform $\widehat{w}_2(y) := \int_{-\infty}^{\infty} w_2(x) e^{-2\pi i x y}  dx\geq 0$ is rapidly decreasing.  
 
Fix some small $\eta > 0$, and let as usual $\varepsilon > 0$ denote an arbitrarily small constant, not necessarily the same at each occurrence.  The first key observation is that by the standard divisor bound we have  
$$
S(M, T, \alpha) \gg M^{-\varepsilon} \widetilde{S}(M, T, \alpha), 
$$
where
$$
\widetilde{S}(M, T, \alpha) := \# \Big \{ n_i \asymp M^{\eta / 4} , m_i \asymp M^{1 - \eta/4} : \Big | \frac{n_1 m_1 n_2 m_2}{n_3 m_3 n_4 m_4} - \alpha \Big | \ll \frac{1}{T} \Big \} \;.
$$
Denoting $\beta = \log \alpha$, we see that $\widetilde{S}(M, T, \alpha)$ is bounded from below by 
%$$\mathcal{S}(M, T, \alpha) \geq \sum_{n_1, n_2, n_3, n_4}  w_2\left(T \Bigl(\log \frac{n_1n_2}{n_3n_4} - \beta \Bigr)\right) \prod_{j=1}^4 w_1\left(\frac{n_j}{M}\right) . $$
\begin{displaymath}
\begin{split}
&\sum_{\substack{n_1, n_2, n_3, n_4 \\ m_1, m_2, m_3, m_4}} w_2 \Big ( T \Big ( \log \frac{n_1 m_1 n_2 m_2}{n_3 m_3 n_4 m_4} - \beta \Big ) \Big ) \prod_{j = 1}^{4} w_1 \Big ( \frac{n_j}{M^{\eta/4}} \Big ) w_1 \Big ( \frac{m_j}{M^{1 - \eta/4}} \Big )\\
= & \frac{1}{T} \int_{-\infty}^{\infty} \widehat{w_2}(y / T) e^{-2\pi i y \beta} \Big | \sum_{n} w_1 \Big ( \frac{n}{M^{\eta/4}} \Big ) n^{2\pi i y} \Big |^4   \Big | \sum_{m} w_1 \Big ( \frac{m}{M^{1 - \eta/4}} \Big ) m^{2\pi i y} \Big |^4 dy 
\\
=: & I_1(\beta) + I_2(\beta)
\end{split}
\end{displaymath}
%Let 
%$$\widehat{w}_2(y) := \int_{-\infty}^{\infty} w_2(x) e^{-2\pi i x y}  dx$$
%denote the Fourier transform of $w_2$. 
%Then the right hand side equals
%$$\frac{1}{T} \int_{-\infty}^{\infty} \widehat{w}_2(y/T) e^{-2\pi i y\beta} \Bigl|\sum_n w_1 \left(\frac{n}{M}\right) n^{2\pi iy}\Bigr|^4 dy =: I_1(\beta) + I_2(\beta),$$
say, where $I_1(\beta)$ is the integral restricted to $|y| \leq M^{\varepsilon}$ for some very small fixed $\varepsilon > 0$, and $I_2(\beta)$ is the rest. 

Let 
$$\check{w}_1(s) := \int_{0}^{\infty} w_1(x) x^{s}  \frac{dx}{x}$$
denote the Mellin transform of $w_1$. Then  
\begin{equation}\label{mellin}
\Sigma(N, y) := \sum_n w_1 \left(\frac{n}{N}\right) n^{2\pi iy} = \int_{(2)} \check{w}_1(s) N^s \zeta(s - 2\pi i y) \frac{ds}{2\pi i}.
\end{equation}
To analyze $I_1(\beta)$ we  shift the contour in \eqref{mellin} to $\Re s = 0$, say, and using the rapid decay of $\check{w}_1$ along vertical lines, we see that for $|y| \leq M^{\varepsilon}$ and $N=M^c$ ($c=\eta/4$ or $1-\eta/4$) we have 
\begin{equation}\label{mellin1}
\Sigma(N, y) = \check{w}_1(1 + 2 \pi i y) N^{1 + 2\pi i y} + O(N^{\varepsilon}).
\end{equation}
%On the other hand, if $|y| \geq M^{\varepsilon}$, we may shift the contour to $\Re s = 1/2$; the pole at $s = 1 + 2\pi i y$ is negligible due to the rapid decay of $\check{w}_1$, and hence we obtain the upper bound
%\begin{equation}\label{mellin2}
%  \Sigma(M, y) \ll  M^{1/2} \int_{-\infty}^{\infty } \frac{|\zeta(1/2 - 2\pi i y + i t)| }{1 + |t|^{10}} dt.
%\end{equation}
%We will use \eqref{mellin1} and \eqref{mellin2} to analyze $I_1(\beta)$ and $I_2(\beta)$, respectively. 
From \eqref{mellin1} we conclude (using also Taylor's theorem in the second step)
\begin{equation*}%\label{i1}
\begin{split}
I_1(\beta) & = \frac{M^4}{T} \int_{- M^{\varepsilon}}^{ M^{\varepsilon}}\widehat{w}_2\left(\frac yT\right) e^{-2\pi i y\beta} | \check{w}_1(1 + 2 \pi i y) |^8 dy  + O \left(\frac{M^{4-\frac \eta 4+\varepsilon} }{T}\right)\\
& =c(\beta)  \frac{M^4}{T}   + O \left(\frac{M^{4-\frac \eta 4 +\varepsilon} }{T} + \frac{M^4}{T^2} \right),
\end{split}
\end{equation*}
where 
$$c(\beta) =\widehat w_2(0) \int_{- \infty}^{\infty}  e^{-2 \pi i y\beta }  | \check{w}_1(1 + 2 \pi i y) |^8 dy.$$
If we define $v(t) := w_1(e^t) e^t$, again a non-negative compactly supported function, then $\check{w}_1(1+ 2 \pi i y) = \widehat{v}(- y)$, so that
%\marginpar{Check!} 
$$c(\beta) = \widehat w_2(0) \int_{\Bbb{R}^7} v(t_1) v(t_2)\dots  v(t_7) v ( -\beta + t_1 + t_2+t_3+t_4 - t_5-t_6-t_7) dt_1\cdots dt_7.$$
If the support of $w_1$ is sufficiently large, then $c(\beta)$  is bounded away from 0, uniformly  for all $e^{\beta} \in \mathcal{J}$, so that
$$I_1(\beta) \gg \frac{M^4}{T}$$
uniformly in $\beta$.

It remains to show that for almost all $\beta$ the contribution $I_2(\beta)$ of the large frequencies $|y| > M^{\varepsilon}$ is of lower order of magnitude. Let 
%We now turn our attention to $I_2(\beta)$ and define
$$\mathcal{I}  := \Bigl(\int_{\log \mathcal{J}} |I_2(\beta)|^2 d\beta\Bigr)^{1/2}.$$
Suppose we can show 
\begin{equation}\label{suppose}
\mathcal{I} \ll M^{4-\rho} T^{-1}
\end{equation}
 for $T = M^{4-\eta}$ and $\rho  > 0$ possibly depending on $\eta$.   Then we conclude $\mathcal{I}_2(\beta) \ll M^{4-\rho/2} T^{-1}$ for all $\beta$ except for a small set $\mathcal{T}_M$ of measure $\ll M^{-\rho}$, so that for all $\alpha \in \mathcal{J} \setminus \mathcal{T}_M$, 
\begin{align*}
S(M, M^{4-\eta}, \alpha) & \gg M^{-\varepsilon} \cdot \widetilde{S}(M, M^{4-\eta}, \alpha) \\ & \gg 
M^{-\varepsilon} \cdot \Big ( \frac{M^4}{T} + O \Big ( \frac{M^{4 - \rho}}{T} \Big ) \Big )  \gg \frac{M^{4 - \varepsilon}}{T} \geq 1
\end{align*}
  and the proof of Proposition \ref{prop9.1} is complete.

To bound $\mathcal{I} $, we note that $I_2(\beta)=\widehat F(\beta)$ is the Fourier transform of 
$$ F(y):=\mathbf 1(|y|>M^\varepsilon)\,\frac{1}{T}  \widehat w_2\left(\frac yT\right) \big|\Sigma(M^{\frac \eta 4}, y )\Sigma(M^{1-\frac \eta 4},y ) \big|^4.$$   
Therefore
$$
\mathcal{I}^2 \leq \int_{-\infty}^\infty |I_2(\beta)|^2d\beta  = \int_{-\infty}^\infty  |F(y)|^2dy
$$
by Plancherel, that is 
\begin{equation} \label{main}
\mathcal{I}^2 %& := \int_{\log \mathcal{J}} |I_2(\beta)|^2 d \beta \\ & \nonumber 
\ll \frac{1}{T^2} \int_{|y| > M^{\varepsilon}} \Big|\widehat{w_2}  \left( \frac{y}{T} \right) \Big |^2 
\Big | \Sigma(M^{\frac \eta 4},y) \Big|^8    \Big|\Sigma(M^{1-\frac \eta 4},y) \Big|^8
%\sum_{n} w_1 \Big ( \frac{n}{M^{\eta/4}} \Big ) n^{2\pi i y} \Big |^8 \cdot 
%\Big | \sum_{m} w_1 \Big ( \frac{m}{M^{1 - \eta/4}} \Big ) m^{2\pi i y} \Big |^8 
dy. 
\end{equation}
%\marginpar{I squared $\widehat w_2$; we don't need positivity} 
 
Since $|y| \geq M^{\varepsilon}$, we may bound $\Sigma(M^{\frac \eta 4},y)$ by shifting the contour in \eqref{mellin} 
to $\Re s = 1-\eta^4$; now the pole at $s = 1 + 2\pi i y$ is negligible due to the rapid decay of $\check{w}_1$, and hence we obtain the upper bound
 $$
  \Sigma(M^{\frac \eta 4},y)  
%\Big | \sum_{n} w_1 \Big ( \frac{n}{M^{\eta/4}} \Big ) n^{2\pi i y} \Big | 
\ll M^{(1 - \eta^4) \eta/4} \int_{-\infty}^{\infty} \frac{|\zeta(1 - \eta^4 - 2\pi i y + it)|}{1 + |t|^{10}} dt.
$$
The crucial input is now a  bound of the type 
$$|\zeta(\sigma + it)| \ll   |t|^{A (1 - \sigma)^{3/2} + \varepsilon}, \quad 1/2 \leq \sigma \leq 1, \, |t| \geq 2$$
where both $A$ and the implied constant are absolute. A first result of this type was proved by Richert \cite{Ri} with $A = 100$; recently Heath-Brown \cite{HB} obtained $A = 1/2$.  %For fixed $\sigma < 1$, we can obviously remove the logarithmic factor at the cost of changing 100 to 101. 
%Using the bound $|\zeta(\sigma + it)| \ll (1 + |t|)^{C (1 - \sigma)^{3/2}}$ \marginpar{reference?} \footnote{in the second edition of Titchmarsh \S 6.19 page 135 they quote H. E. Richert, Zur Absch\"atzung der Riemannschen Zetafunktion in der N\"ahe der Vertikalen $\sigma=1$, Math. Ann. 169 (1967), 97--101. with $C=100$ (and an extra factor of $(\log t)^{2/3}$) uniformly in $0\leq \sigma\leq 2$, $t\geq 2$. In turn, Richert quotes Stas who gave a constant of $C=2^{15}$ with $1-\frac 1{2^{13}} <\sigma<1$, $t\geq 3$, but apparently the implied constant depends on $\sigma$. } 
%(which is non-trivial near the line $\sigma=1$), 
We conclude that 
$$
 \Big| \Sigma(M^{\frac \eta 4},y) \Big| \ll M^{(1 - \eta^4) \eta / 4} \cdot |y|^{A \eta^{6} + \varepsilon},
$$
%for  an absolute constant $c > 0$, 
so that   
   \eqref{main}  is  bounded by 
\begin{align*}
& \ll M^{2 (1 - \eta^4) \eta} \cdot \frac{1}{T^2} \int_{\mathbb{R}}\Big|\widehat{w_2} \Big ( \frac{y}{T} \Big )\Big|^2 |y|^{8A \eta^6+\varepsilon}  \Big|\Sigma(M^{1-\frac \eta 4},y) \Big|^8
%\Big | \sum_{m} w_1 \Big ( \frac{m}{M^{1 - \eta/4}} \Big ) m^{2\pi i y} \Big |^8 
dy 
\\
&\ll  M^{2 (1 - \eta^4) \eta} \cdot \frac{1}{T^2} \cdot T^{8A \eta^6+\varepsilon} \cdot \int_{|y|\leq T^{1+\varepsilon}} \Big| \sum_{n\ll M^{4-\eta}} a(n) n^{2\pi i y} \Big|^2 dy,
\end{align*}
where 
$$
a(n) = \sum_{n_1\cdot \ldots \cdot n_4=n} w_1\Big(\frac {n_1}{M^{1-\frac \eta 4}}\Big) \cdot  \ldots \cdot w_1\Big(\frac {n_4}{M^{1-\frac \eta 4}}\Big) \ll n^{\varepsilon} \;.
$$
Using the standard mean value theorem \cite[Theorem 9.1]{IK} 
$$\int_0^X \Bigl| \sum_{n \leq N} a_n n^{it} \Bigr|^2 dt \ll (X + N) \sum_{n \leq N} |a_n|^2,$$
%and the  divisor bound $a(n)\ll n^\varepsilon$
we obtain for $T = M^{4-\eta}$ that 
\begin{align*}
\mathcal{I}^2 & \ll M^{2 (1 - \eta^4) \eta +\varepsilon} \cdot \frac{1}{T^2} \cdot T^{8A \eta^6} \cdot \Big ( T + M^{4 - \eta} \Big ) \cdot M^{4 - \eta  } 
\\&  
\ll \frac{M^{8 - 2 \eta^5 + 32A \eta^6+\varepsilon}}{T^2} \ll \frac{M^{8 - \eta^5}}{T^2} 
\end{align*}
for all sufficiently small $\eta > 0$. % small enough and where we used that $T = M^{4 - \eta}$. % (one can in fact take $c = 400$ and then we see that
%$\eta \leq 10^{-4}$ works). \marginpar{Reference?}
This shows \eqref{suppose} with $\rho = \frac{1}{2} \eta^5$ and completes the proof of Proposition \ref{prop9.1}. 
Moreover Heath-Brown's result \cite{HB} allows us to pick $A = \tfrac 12$ in which case any $0 < \eta < 1/16$ is admissible.

%So it follows that for almost all $\beta$ (except for a subset of measure $\ll M^{-\eta^5/2}$) 
%we have $|I_2(\beta)| \ll M^{4 - \eta^5 / 4} T^{-1}$. On the other hand recall that 
%$|I_1(\beta)| \gg M^4 / T$. It follows that for all $\alpha \in \mathcal{J}$ (with the exception of a subset of measure  $\ll M^{-\eta^5/2}$ of exceptions) we have
%$$
%\widetilde{S}(M, T, \alpha) \gg \frac{M^4}{T}
%$$
%And therefore, for almost all $\alpha$ (except for a subset of measure $\ll M^{-\eta^5/2}$),
%$$
%S(M, T, \alpha) \gg M^{-\varepsilon} \cdot \frac{M^4}{T} \geq M^{\eta - \varepsilon} \geq 1
%$$
%since $\eta > 0$ is fixed.

\section{Proof of Theorem \ref{thm:nonpoisson}}
  One simple reason why the sequence of eigenvalues is non-generic as far as the behaviour of minimal gaps is concerned, is that it is closed under multiplication by perfect squares, hence one small gap propagates. Indeed, 
let $\alpha > 0$ be arbitrary, and suppose \eqref{dev1} holds, that is $\delta_{\min}^{(\alpha)}(N) \leq \frac{1}{N \log N}$  infinitely often.  
Let $\lambda', \lambda'\ll N$  be two eigenvalues with $$0 < \lambda - \lambda' \leq \frac{1}{ N \log N}.$$ Then obviously $\tilde{\lambda} := 4\lambda$, $\tilde{\lambda}' := 4\lambda'$ are  eigenvalues with
 $$0 < \tilde{\lambda} - \tilde{\lambda'} \leq  \frac{4}{ N \log N},$$
so that 
$$\delta^{(\alpha)}_{\min, 2}(cN)  \leq \frac{4}{N \log N} \quad\quad \text{infinitely often}$$
for some suitable constant  $c$, violating \eqref{dev1a}.


\begin{thebibliography}{99}
 
 \bibitem{BB}
G.  Ben Arous,  P. Bourgade, 
{\em Extreme gaps between eigenvalues of random matrices}, 
The Annals of Probability
2013, Vol. 41, No. 4, 2648--2681.


\bibitem{BT}
M. V. Berry,  M. Tabor, {\em Level clustering in the regular spectrum}, Proc. Roy. Soc. London A 356 (1977), 375--394.

\bibitem{BGS}
O. Bohigas, M.-J. Giannoni,  C. Schmit, {\em Spectral fluctuations of classically chaotic quantum systems}, in ``Quantum Chaos and Statistical Nuclear Physics", edited by Thomas H. Seligman and Hidetoshi Nishioka, Lecture Notes in Physics Vol. 263 (Springer-Verlag, Berlin, 1986), p. 18--40.

\bibitem{Bourgain} J. Bourgain, \emph{A quantitative Oppenheim theorem for generic diagonal quadratic forms}, {\tt arXiv:1604.02087}

%\bibitem{BDG} J. Bourgain, C. Demeter, L. Guth, \emph{Proof of the main conjecture in Vinogradov�s mean value theorem for degrees higher than three}, {\tt arXiv:1512.01565}. 

%\bibitem{BLMS}
%Y. Bugeaud, F. Luca, M. Mignotte and S. Siksek, {\em On Fibonacci numbers with few prime divisors}. Proc. Japan Acad. Ser. A Math. Sci. 81 (2005), no. 2, 17--20.

%\bibitem{BG}
%V. Blomer, A. Granville, \emph{Estimates for representation numbers of quadratic forms}, Duke Math. J. \textbf{135} (2006), 261-302

%\bibitem{Conrey Gonek}
%J. B. Conrey,  S. M.  Gonek, 
%{\em High moments of the Riemann zeta-function}.  
%Duke Math. J. 107 (2001), no. 3, 577--604. 

\bibitem{Devroye} L. Devroye, \emph{Upper and lower class sequences for minimal uniform spacings}, Zeitschr. Wahrsch. verw. Geb.  61 (1982), 237-254.

%\bibitem{EY}
%L. Erd\H{o}s  and H.-T. Yau, {\em Gap universality of generalized Wigner and $\beta$-ensembles}. J. Eur. Math. Soc. (JEMS) 17 (2015), no. 8, 1927--2036. 

%\bibitem{Erdos}
%P. Erd\H{o}s, 
%{\em On the distribution of the convergents of almost all real numbers.}  
%J. Number Theory 2 1970 425--441.

\bibitem{Er1}
P. Erd\H{o}s, \emph{Some remarks on number theory}, Riveon Lematematika 9 (1955), 45-48
 
\bibitem{EMM}
A. Eskin, G. Margulis,  S. Mozes, {\em Quadratic forms of signature (2,2) and eigenvalue spacings on rectangular 2-tori}. Ann. of Math. (2) 161 (2005), no. 2, 679--725. 

%\bibitem{E}
%L. Euler,  De fractionibus continuis dissertatio, Comm. Acad. Sci. Petropol. 9 (1744) 98-137, also in Opera Omnia, ser. I, vol. 14, Teubner, Leipzig, 1925, 87-215

\bibitem{Fo}
K. Ford, \emph{The distribution of integers with a divisor in a given interval}, Annals of Math. (2) 168 (2008), 367-433

%\bibitem{HB12} 
%D. R. Heath-Brown, \emph{The twelfth power moment of the Riemann-function}, 
%Quart. J. Math. Oxford \textbf{29} (1978),   443-462.

\bibitem{HB}
D. R. Heath-Brown, \emph{A new $k$-th derivative estimate for exponential sums via Vinogradov's mean value}, {\tt arXiv:1601.04493}

\bibitem{HS}
P. Horak,  L. Skula,  
{\em A characterization of the second-order strong divisibility sequences. } 
Fibonacci Quart. 23 (1985), no. 2, 126--132.
  
  
  \bibitem{IK} H. Iwaniec, E. Kowalski, \emph{Analytic Number Theory}, Colloquium Publications \textbf{53} (2004), American Math. Soc., Providence, RI.

%\bibitem{Ju} M. Jutila, \emph{Zero-density estimates for $L$-functions}, Acta Arith. \textbf{32} (1977), 55-62

%\bibitem{Kappeler}
%T. Kappeler, {\em On double eigenvalues of Schr\"odinger operators on two-dimensional tori}. J. Funct. Anal. 115 (1993), no. 1, 166--183.
 
%\bibitem{Keating Snaith}
%J. P. Keating, N. C.  Snaith, {\em Random matrix theory and} $\zeta(1/2+it)$. Comm. Math. Phys. 214 (2000), no. 1, 57--89.

\bibitem{Khinchine}
A. Ya.  Khinchin,  {\em Continued fractions}. With a preface by B. V. Gnedenko. Translated from the third (1961) Russian edition. Reprint of the 1964 translation. Dover Publications, Inc., Mineola, NY, 1997. 

\bibitem{Knuth}
D. E. Knuth, {\em The art of computer programming}. Vol. 1. 
Fundamental algorithms. Third edition. Addison-Wesley, Reading, MA, 1997.

%\bibitem{Landau}
%E.  Landau, {\em \"Uber die Einteilung der positiven ganzen Zahlen in vier Klassen nach der ¨ Mindeszahl der zu ihrer additiven Zusammensetzung erforderlichen Quadrate}. Arch. Math. Phys., 13:305--312, 1908.


\bibitem{Levy}
P.  L\'evy, {\em Sur la division d'un segment par des points choisis au
hasard}, C.R. Acad. Sci. Paris 208 (1939), 147--149.

\bibitem{Montgomery}
H. L. Montgomery, 
{\em The pair correlation of zeros of the zeta function}. 
Analytic number theory (Proc. Sympos. Pure Math., Vol. XXIV, St. Louis Univ., St. Louis, Mo., 1972), pp. 181--193. Amer. Math. Soc., Providence, R.I., 1973. 

%\bibitem{Norfleet}
%M. Norfleet, {\em Characterization of second-order strong divisibility sequences of polynomials}. 
%Fibonacci Quart. 43 (2005), no. 2, 166--169. 

\bibitem{RTW}
 M. O. Rayes, V. Trevisan,   P. S. Wang, {\em Factorization properties of Chebyshev polynomials}. 
Comput. Math. Appl. 50 (2005), no. 8-9, 1231--1240.

\bibitem{Ri}
 H. E. Richert, {\em Zur Absch\"atzung der Riemannschen Zetafunktion in der N\"ahe der Vertikalen $\sigma=1$}, Math. Ann. 169 (1967), 97--101.

\bibitem{Rivlin}
T. Rivlin, {\em  Chebyshev polynomials: from approximation theory to  algebra and number theory}, Wiley, New York 1990

\bibitem{RS}
 Z.Rudnick, P. Sarnak, {\em Zeros of principal L-functions and random matrix theory}. A celebration of John F. Nash, Jr. Duke Math. J. 81 (1996), no. 2, 269--322. 
 
\bibitem{Sarnak}
P. Sarnak, 
{\em Values at integers of binary quadratic forms}. 
Harmonic analysis and number theory (Montreal, PQ, 1996), 181--203, 
CMS Conf. Proc., 21, Amer. Math. Soc., Providence, RI, 1997.

%\bibitem{Te}


\bibitem{Vinson}
J. Vinson,  {\em Closest spacing of eigenvalues}, Ph.D. thesis Princeton 2001, {\tt arXiv:1111.2743} %[math.SP]

%\bibitem{Wooley}
%T. D. Wooley, {\em The cubic case of the main conjecture in Vinogradov's mean value theorem}, 	{\tt arXiv:1401.3150}. % [math.NT]. 

\end{thebibliography}
\end{document}